\documentclass[11 pt]{article}

\usepackage[utf8]{inputenc}
\usepackage[english]{babel}   
\usepackage{graphicx}

\usepackage{lmodern}
\usepackage{amsmath}
\usepackage{tikz}
\usetikzlibrary{arrows.meta,calc,decorations.pathreplacing}

\usepackage{longtable}
\usepackage{epsfig}
\graphicspath{{figures/}{./}}
\usepackage{color}
\usepackage{xcolor}
\usepackage[normalem]{ulem}
\usepackage{multicol}
\usepackage{setspace}
\usepackage[top=2cm,bottom=2.5cm,left=2.5cm,right=2.5cm]{geometry}
\usepackage{sectsty}
\usepackage{fancyhdr}
\usepackage{wasysym}
\usepackage[colorlinks=true]{hyperref}
\usepackage{subcaption}
\usetikzlibrary{patterns}
\usepackage{amsfonts}
\usepackage{amssymb} 
\usepackage{amsthm}
\usepackage{bm}
\usepackage[normalem]{ulem}

\renewcommand{\div}{\mathrm{div}}
\renewcommand{\skew}{\mathrm{skew}}

\newtheorem{theorem}{Theorem}

\newtheorem{prop}[theorem]{Proposition}
\newtheorem{remark}[theorem]{Remark}

\title{Computation of stresses in jammed packings modeled with Tresca friction}
\author{ \begin{minipage}{\textwidth}\centering
		Fr\'ed\'eric Marazzato$^1$ and Shankar Venkataramani$^2$ \\
		\small{$^1$Department of Mathematical Sciences, University of Nevada Las Vegas, Las Vegas, NV 89154-4020, USA}\\
		\small{$^2$Department of Mathematics, The University of Arizona, Tucson, AZ 85721-0089, USA}\\
   \small{email: \texttt{frederic.marazzato@unlv.edu}, \texttt{shankar@arizona.edu}\\}
   \end{minipage}}
      
\date{}
   
\begin{document}
\hypersetup{urlcolor=blue,linkcolor=red,citecolor=blue}

\maketitle

\begin{abstract}
This paper is interested in the computation of stresses within jammed packings of rigid polygonal cells.
The cells are considered to follow a Tresca friction law.
First, a constrained minimization problem is introduced where the friction energy is minimized while enforcing the non-interpenetration of neighboring cells as inequality constraints.
The corresponding dual maximization problem is then deduced and its solutions provide normal stresses at the interface between cells.
Finally, lowest order Raviart--Thomas finite elements are used to reconstruct a consistent stress field by solving local problems.
Numerical results are presented to showcase the consistency and robustness of the proposed methodology.
\end{abstract}

\section{Introduction}
Jammed particulate systems, such as granular materials, or masonry structures, exhibit complex mechanical behaviors arising from a disordered microstructure and a network of inter-particle contacts.
When subjected to external loads, these systems support forces through a self-organized arrangement of contact stresses that ensure mechanical equilibrium.
Understanding and accurately computing the stress fields within such jammed configurations is essential for predicting macroscopic mechanical response and failure mechanisms.
Typical methods to compute stresses within granular material include Discrete Element Methods (DEM) as in \cite{cundall1979discrete,PhysRevE.72.041307}.

Granular materials inherently contain a lot of voids, which is an extra difficulty on its own.
This paper is concerned with dense, jammed packings of polygonal cells that do not have such voids e.g. brick walls or vaults, see Figure \ref{fig:examples}.
\begin{figure}[!htp]
\centering
\begin{subfigure}[b]{0.45\textwidth}
        \centering
        \includegraphics[scale=.15]{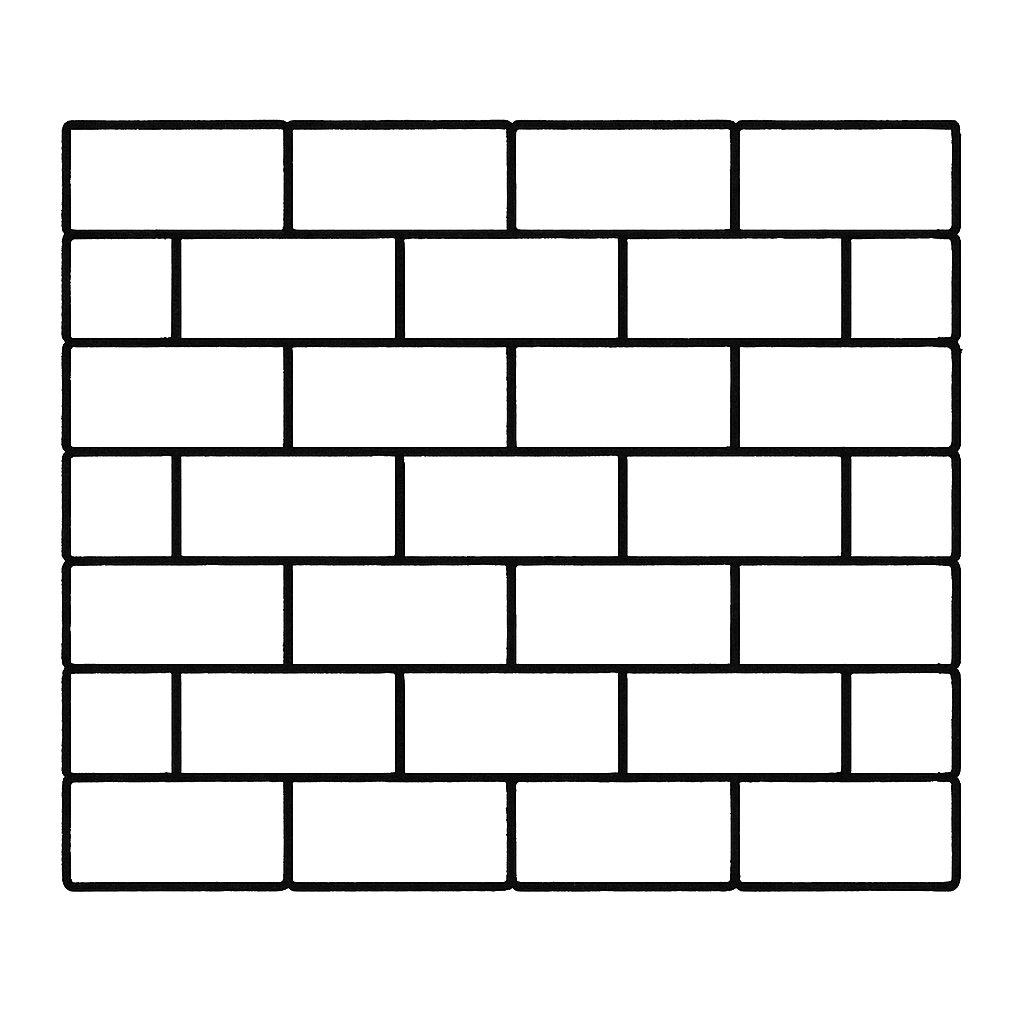}
        \caption{Brick wall}
    \end{subfigure}
    \hfill
    \begin{subfigure}[b]{0.45\textwidth}
        \centering
        \includegraphics[scale=.1]{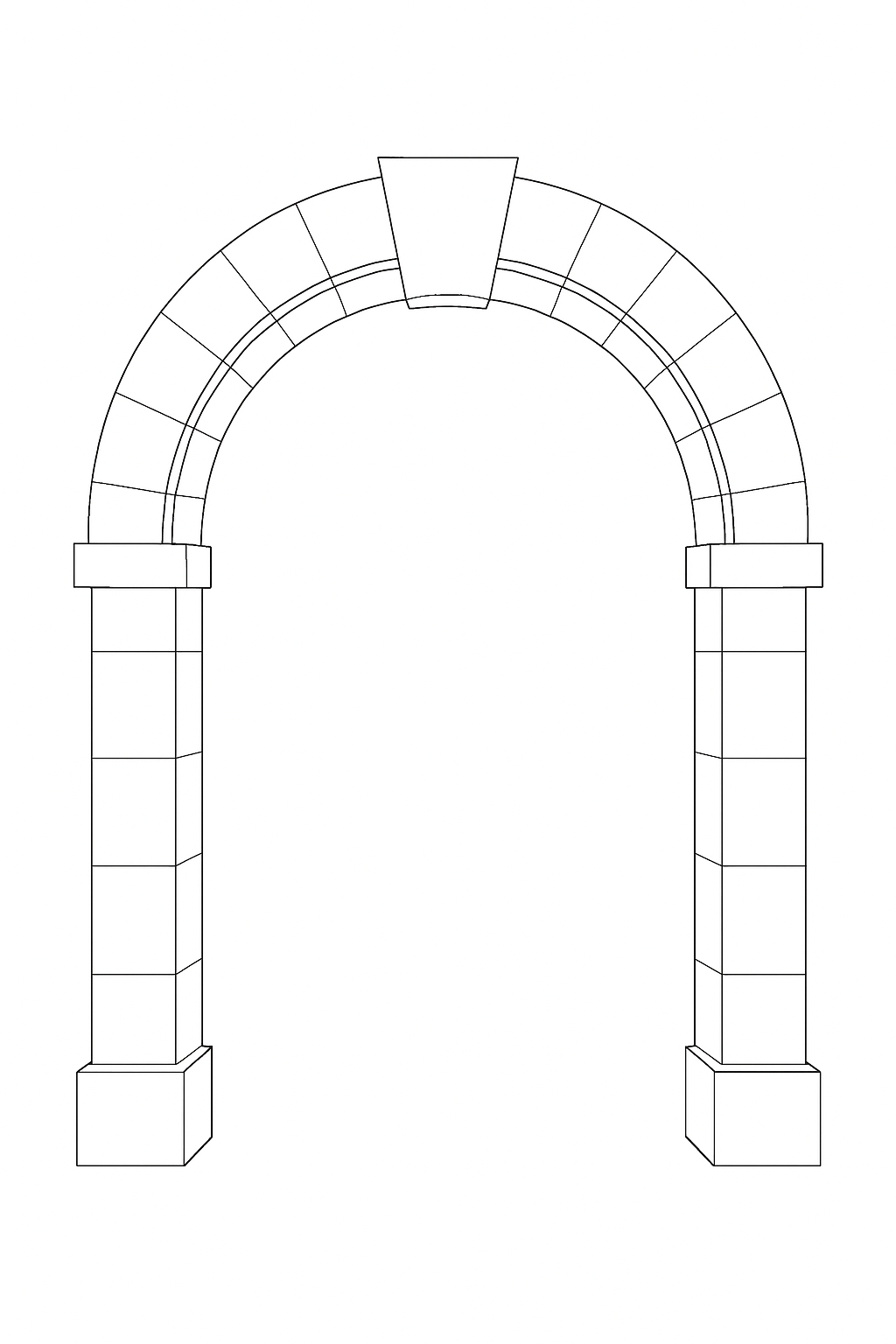}
        \caption{Vault with a keystone}
    \end{subfigure}
    \caption{Examples of jammed packings}
\label{fig:examples}
\end{figure}
Let us denote by $\Omega \subset \mathbb R^d$, where $d=2,3$ a domain of interest (the jammed packing).
We will consider that a packing is defined by a mesh of convex polygonal cells, denoted by $\mathcal M$, see Figure \ref{fig:voronoi}.
\begin{figure}[htp!]
\centering
\includegraphics[scale=.3]{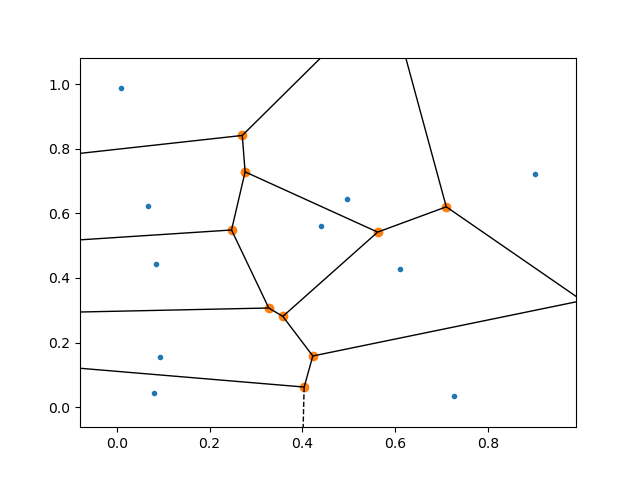}
\caption{Polygonal mesh $\mathcal M$ of $\Omega$.}
\label{fig:voronoi}
\end{figure}
Throughout this paper, we will refer to each element of the packing as a cell (of the mesh $\mathcal M$).
In Figure  \ref{fig:voronoi}, the blue dots represent the appropriate center for the cell (centroid for generic convex cells, Voronoi center for Voronoi cells, etc.).
The interfaces between cells will then be the internal edges of the mesh $\mathcal M$.

For certain engineering applications (civil engineering in the case of Figure \ref{fig:examples}), the range of external loads that jammed packings are subjected to does not lead to a significant elastic deformation of the cells.
Therefore, one can consider that the cells are fully rigid (i.e. they do not deform elastically).
When the packing is submitted to external loads, even though it might not deform elastically, internal stresses are still present within $\Omega$.
Under certain conditions, masonry structures \cite{roca2010structural} fall into that category.
The stresses within such jammed packings can then be determined exclusively by friction forces.
As \cite{tralli2014computational} puts it, masonry structures are not accurately modeled by elasticity.
We then decide to model the physics of jammed packings by the rigid friction of neighboring cells with a Tresca friction criterion.
Plasticity in the joint and other physical phenomena \cite{roca2010structural} should be taken into account to produce a more physically realistic model.
This work proposes a simplified model that allows to deduce stresses in jammed packings only from friction, without elastic effects.

Since \cite{cundall1979discrete}, many DEM have computed friction forces by penalizing interpenetration.
The friction forces are then deduced linearly from the interpenetration.
Plasticity can also be taken into account in \cite{cundall1979discrete}.
More recently, \cite{PORTIOLI2017485} has proposed an approach, using DEM, that does not allow interpenetration of neighboring cells and does not use elastic parameters.
Within appropriate loading regimes, we believe this approach better captures the underlying physics of jammed packings.
The main difference with our work is that \cite{PORTIOLI2017485} models large displacements whereas we are interested in the stresses within jammed packings.

This paper will first give a short reminder on Tresca friction.
It will then present the reconstructions used to obtain a primal energy based on the Tresca criterion.
The existence of minimizers of the primal energy is shown under appropriate hypotheses.
Assuming the jammed packing is stable, the displacement of each cell should be zero (up to machine error).
Thus, only solving the primal problem is not useful for the task at hand.
The dual program is then introduced and classical results relating the primal and dual problem are recalled.
The output of solving the dual problem is a collection of normal stresses at the edges of the mesh $\mathcal M$.
Using this data, a consistent reconstruction of the stresses, is then presented.
The reconstruction takes advantage of the fact that the lowest-order Raviart--Thomas elements $\mathbb{RT}_0$ are $H(\div)$ conforming and that the associated degrees of freedom (dofs) are the normal stresses at the barycenter of the edges of the mesh.
The idea is then to reconstruct in each cell $c \in \mathcal M$ a stress $\sigma_c \in \mathbb R^{d \times d}_\text{sym}$, the space of symmetric square tensors of size $d$, such that $\mathrm{div}(\sigma_c) = 0$ (in the absence of volume loads).
Finally, numerical results are presented to show the consistency of the proposed methodology as well as its efficiency in computing the stresses for domains with numerous cells.

For simplicity, we consider in the rest of this paper that $d=2$.
Extending the following methodology to $d=3$ is straightforward.

\subsection{Tresca Friction}

Let us illustrate the Tresca friction law in a very simple case.
One can refer to \cite{chouly2023finite} for further detail.
Let $\Omega \subset \mathbb R^2$ be an open bounded domain.
We assume that the boundary of $\Omega$, written $\partial \Omega$, is Lipschitz and write as $\Gamma_c \subset \partial \Omega$ the contact surface.
We write as $\bm n$ the outward unit normal to $\partial \Omega$ and as $\bm t = R(\frac\pi2) \bm n$ the unit tangent to $\partial \Omega$, where $R(\frac\pi2)$ is the 2d rotation matrix by an angle $\frac\pi2$.
We write as $\bm u:\Omega \to \mathbb R^2$ the displacement field and $\sigma: \Omega \to \mathbb R^{2 \times 2}_\text{sym}$ the stress tensor.
We consider that the material in $\Omega$ is potentially in contact at $\Gamma_c$ with a fully rigid material (it does not deform).
See figure \ref{fig:sketch friction} for a sketch of the setup.
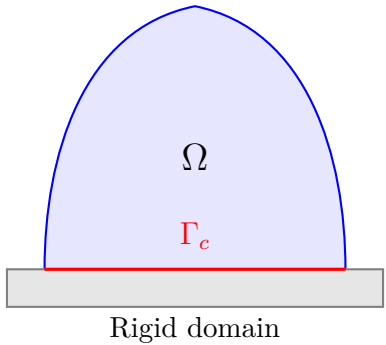
\begin{figure}[!htp]
\centering
\begin{tikzpicture}

\filldraw[fill=blue!10, draw=blue, thick]
    (0,0)
    .. controls (0,1.5) and (0.5,3.2) .. (2,3.5)
    .. controls (3.5,3.2) and (4,1.5) .. (4,0)
    .. controls (4,-0.1) and (0,-0.1) .. (0,0);  

\node at (2,1.5) {\Large $\Omega$};

\draw[gray, thick, fill=gray!20] (-0.5,-0.5) -- (4.5,-0.5) -- (4.5,0) -- (-0.5,0) -- cycle;
\node[below] at (2,-0.5) {Rigid domain};

\draw[red, very thick] 
    (0,0)
    .. controls (4/3,0) and (8/3,0) .. (4,0);

\draw[red] (0,0) -- (4,0) node[midway,above=4pt,red] {\large $\Gamma_c$}; 
\end{tikzpicture}
\caption{Sketch of the contact between a rigid domain and $\Omega$.}
\label{fig:sketch friction}
\end{figure}

Let us decompose the displacement and normal stress on $\Gamma_c$ as $\mathbf u = u_{\mathbf n} \mathbf n + \mathbf u_{\mathbf t}$ and $\sigma \mathbf n = \sigma_{\mathbf n} \mathbf n + \bm \sigma_{\mathbf t}$.
The kinematic contact conditions at $\Gamma_c$ are:
\begin{equation}
\label{eq:contact conditions}
\left\{
\begin{aligned}
&u_{\bm n} \le 0, \\
&\sigma_{\bm n} \le 0, \\
& \sigma_{\bm n} u_{\bm n} = 0. \\
\end{aligned}
\right.
\end{equation}
The first equation in \eqref{eq:contact conditions} prevents the penetration of $\Omega$ into the rigid domain, whereas the second describes that contact forces can only be repulsive.
The third equation describes that two situations are possible at any point on $\bm x \in \Gamma_c$:
\begin{itemize}
\item if the domain $\Omega$ is detached from the rigid support $u_{\bm n}(\bm x) < 0$, then necessarily $\sigma_{\bm n} = 0$
(there cannot be a contact force),
\item if there is a negative contact force $\sigma_{\bm n}< 0$, then necessarily the elastic body must stick to the support $u_{\bm n} = 0$.
\end{itemize}
Note that \eqref{eq:contact conditions} are sometimes called KKT conditions in the context of constrained optimization and in that context the normal stress $\sigma_{\bm n}$ is the Lagrange multiplier of the contact constraint.
The equations in \eqref{eq:contact conditions} are common to all contact laws and not specific to the Tresca friction law.
The Tresca friction law writes on $\Gamma_c$ as
\begin{equation}
\label{eq:tresca criterion}
\left\{ \begin{aligned}
& |\bm \sigma_{\mathbf t}| \le s_T &\text{if } \mathbf u_{\bm t} = \bm 0, \\
& \bm \sigma_{\bm t} = -s_T \frac{\bm u_{\bm t}}{|\bm u_{\bm t}| } &\text{otherwise},
\end{aligned} \right.
\end{equation}
where $s_T > 0$ is the Tresca friction coefficient.
The energy dissipation due to the Tresca friction writes as 
\begin{equation}
\label{eq:energy tresca}
\int_{\Gamma_c} s_T|\bm u_{\bm t}|.
\end{equation}

\section{Discrete Model}
\label{sec:model}
We consider an open bounded polygonal domain $\Omega \subset \mathbb{R}^2$.
For simplicity, the domain will be subjected to normal stress (or Neumann) boundary conditions but no volume loads.
Including volume loads would be a straightforward extension.

\subsection{Setup}
Let $\mathcal M$ be a decomposition of a domain $\Omega$ into a collection of convex polygonal cells.
Let us write as $\mathcal{E}$ the set of edges of the mesh.
The edges are partitioned as $\mathcal{E} = \mathcal{E}^\partial \cup \mathcal E^i$, where $\mathcal E^\partial$ are the external edges and $\mathcal E^i$ are the internal edges.
For a cell $c \in \mathcal M$, we denote the position of its center by $\bm x_c$, which also represents its initial position.
Each cell $c \in \mathcal M$ is composed of a collection of edges written $\mathcal{E}_c$.
The cells are considered to be rigid, and thus we take the current position $\bm z \in \mathbb{P}^0(\mathcal M)^2 =: V$ as piecewise constant in each cell.
The initial position of the cells is also represented by a piecewise constant functions $\bm x \in V$.
The displacement $\bm u_c$ of a cell $c \in \mathcal M$ is defined as $\bm u_c := \bm z_c - \bm x_c$.
The collection of cell displacements is then stored in $\bm u \in V$.
Since the cells are considered as rigid bodies, a complete description of the kinematics of a cell $c \in \mathcal M$ would require a translation of a distinguished point $\bm u_c$, and a rotation $\theta_c$, about that point.

However, in our present modeling, we choose to neglect the rotational degrees of freedom.
This simplification is motivated physically by the fact that the rotation of cells in a jammed configuration requires significant dilatancy and, as a consequence, the resistance of compressed packings to rotation is often orders of magnitude higher than the resistance to sliding (friction). This motivates our modeling approach where we solve only for force balance (dual to the translation degrees of freedom) and do not enforce the moment balance equations (which are the dual to the rotational degrees of freedom).

\begin{remark}
\label{rem:general mesh}
In the context of frictionless packings, Voronoi cells are typically used because they automatically satisfy the moment equations. Without friction, the interaction forces are purely normal to the edges. Since the normal bisectors of a Voronoi cell all intersect at the center the net moment is always zero.
For frictional packings, however, there is no particular mechanical advantage to using Voronoi cells over general polygonal cells, since tangential forces are present and moment balance is not automatically satisfied.
We focus on convex cells here because they are common in practice and offer flexibility in choosing the cell center $\bm x_c$.
\end{remark}

\subsection{Hypertstatic packings}
\label{sec:static_indet}
One of the key features of jammed packings is their static indeterminacy (or hyperstaticity).
Let us count the degrees of freedom and constraints for a generic planar packing.
The number of force balance equations depends on whether we enforce torque balance.
Let $n$ be the number of force variables per edge ($n=1$ for frictionless, $n=2$ for frictional).
Let $m$ be the number of moment balance equations per cell ($m=0$ if ignored/automatic, $m=1$ if enforced).
For a system with $F$ cells and $J$ internal edges, the number of degrees of freedom is $nJ$ and the number of constraints is $(2+m)F - 3$ (subtracting 3 global rigid body modes).
We summarize the different modeling choices in Table \ref{tab:counting}.

\begin{table}[h]
\centering
\begin{tabular}{|l|c|c|c|c|}
\hline
\textbf{Model} & $\bm n$ (vars/edge) & $\bm m$ (moments/cell) & \textbf{DoF} & \textbf{Constraints} \\
\hline
Frictionless Voronoi & 1 & 0 (automatic) & $J$ & $2F - 3$ \\
Frictional (Full Equilibrium) & 2 & 1 & $2J$ & $3F - 3$ \\
Frictional (Simplified, this paper) & 2 & 0 & $2J$ & $2F - 3$ \\
\hline
\end{tabular}
\caption{Counting argument for different granular models. $n$: force variables per edge. $m$: moment equations per cell. Our model ($n=2, m=0$) is highly hyperstatic.}
\label{tab:counting}
\end{table}

The condition for isostaticity would be DoF = Constraints.
In our simplified frictional model ($n=2, m=0$), the number of variables $2J$ far exceeds the constraints $2F-3$, leading to a large space of self-equilibrated stresses, corresponding to  the dimension of the null space of the equilibrium matrix, $s = nJ - (2F-3) > 0$.
This means there exists a nontrivial subspace consisting of `states of self-stress' that satisfy force balance with zero external load.
Our optimization approach \eqref{eq:primal problem}, and its numerical implementation effectively selects one specific solution from this hyperstatic manifold, and this state is one among the minimizers of the energy functional \eqref{eq:primal}.

\subsection{Facet values}
In order to write the primal energy that we shall minimize, we need facet values computed from cell values.
For a boundary edge $e \in \mathcal E^\partial$, we write the facet reconstruction operator as $\bm R_e(\bm u)$.
Since $e \in \mathcal E^\partial$, there exists a unique $c \in \mathcal M$ such that $e \subset \partial c$.
The idea of the reconstruction is to simply use the displacement of the cell $c$ that contains the boundary  facet $e$.
Since the packing is jammed, cells cannot spin.
Also, we consider the cells to be fully rigid, which makes this reconstruction consistent.
One thus has
\begin{equation}\label{eq:barycentric reconstruction}
  \bm R_e(\bm u) := \bm u_c.
\end{equation}

For $e \in \mathcal{E}^i$, one defines $\bm n_e$ as the unit normal to $e$ pointing from $c_-$ towards $c_+$, see Figure \ref{fig:sketch}.
The jump operator is defined for $\bm u$ at $e$ by
\[ [\bm u]_e :=
\bm u_{c_+} - \bm u_{c_-}. \]

\begin{figure} [!htp] 
\begin{center}
\begin{tikzpicture}[scale=1.3]
\coordinate (a) at (0,-1);
\coordinate (f) at (0.5,-2.5);
\coordinate (g) at (1.7,0);
\coordinate (i) at (2.2,-1.5);
\coordinate (j) at (3.5,-1.5);
\coordinate (k) at (2.2,-3.5);
\coordinate (l) at (0.8,-4);
\coordinate (p) at (barycentric cs:a=0.25,g=0.25,i=0.25,f=0.25);
\coordinate (q) at (barycentric cs:i=0.2,j=0.2,k=0.2,l=0.2,f=0.2);

\path[draw] (f) -- (a) -- (g) -- (i);
\path[draw] (i) -- (j) -- (k) -- (l) -- (f);

\fill[red] (barycentric cs:a=0.25,g=0.25,i=0.25,f=0.25) circle (2pt);
\draw[left] node at (barycentric cs:a=0.25,g=0.25,i=0.25,f=0.25) {$\bm x_{c_+}$};
\fill[red] (barycentric cs:i=0.2,j=0.2,k=0.2,l=0.2,f=0.2) circle
(2pt);
\draw[right] node at (barycentric cs:i=0.2,j=0.2,k=0.2,l=0.2,f=0.2) {$\bm x_{c_-}$};

\draw[->] (barycentric cs:p=0.5,q=0.5) -- (barycentric cs:p=0.7,q=0.3);
\draw[left] node at (barycentric cs:p=0.7,q=0.3) {$\bm n_e$};
\draw[blue] (f) -- (i);
\draw[blue,below] node at (barycentric cs:p=0.5,q=0.5) {$e$};

\end{tikzpicture}
\caption{Notation for edge quantities}
\label{fig:sketch}
\end{center}
\end{figure}
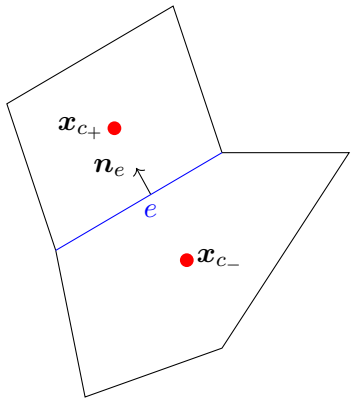
For $e \in \mathcal{E}^b$, one defines $\bm n_e$ as the outward unit normal to $\partial \Omega$.
Finally, for $e \in \mathcal{E}$, we define the tangent vector to $\bm n_e$ as $\bm t_e:= R(\frac\pi2) \bm n_e$, where $R(\frac\pi2)$ is the 2d rotation matrix by an angle $\frac\pi2$.

\subsection{Primal problem}
We consider that the cells interact through Tresca friction forces, and that there can be no inter-penetration between the cells.
Let us first describe the energy $E_e$ between the two neighboring cells $c_{+}$ and $c_{-}$ of Figure \ref{fig:sketch}.
We write the energy of the Tresca friction law (compare to \eqref{eq:energy tresca}) for $e$ as
\begin{equation}
\label{eq:pair energy}
E_e(\bm u) := s_T |e| |[\bm u]_e \cdot \bm t_e|,
\end{equation}
In order to forbid inter-penetration between cells, we add the following constraint:
\begin{equation}
\label{eq:ineq conntraint}
\forall e \in \mathcal E^i, \quad [\bm u]_e \cdot \bm n_e \ge 0.
\end{equation}
We also add a confinement term in order to keep the granular material compressed.
It materializes as a Neumann boundary condition $\sigma \bm n = \bm g$ on $\partial \Omega$.
Assume that $\bm g \in L^2(\partial \Omega)^2$ and let $e \in \mathcal E^\partial$.
We define $\bm g_e := |e|^{-1}\int_e \bm g$ as the average of $\bm g$ over $e$.
Therefore, the work of external loads is defined as
\begin{equation}
E_\text{ext}(\bm u) := \sum_{e \in \mathcal E^\partial} |e| \bm g_e \cdot \bm R_e(\bm u).
\end{equation}
The primal energy $E$ is then defined as
\begin{equation}
\label{eq:primal}
E(\bm u) := \sum_{e \in \mathcal{E}^i}  E_e(\bm u) - E_\text{ext}(\bm u).
\end{equation}
Finally, since \eqref{eq:primal} has a translation invariance, we add the following equality constraint:
\begin{equation}
\label{eq:equality constraint}
\sum_{i=1}^N \bm u_i = 0.
\end{equation}
We therefore define the primal problem as
\begin{equation}
\label{eq:primal problem}
\begin{aligned}
\min_{\bm u \in V} \quad & E(\bm u) \\
\text{subject to} \quad & \forall e \in \mathcal E^i, \quad [\bm u]_e \cdot \bm n_e \ge 0, \\
 & \sum_{i=1}^N \bm u_i = 0.
\end{aligned}
\end{equation}

\subsection{Dual problem}
The problem \eqref{eq:primal problem} is physically relevant but does not contain internal stresses.
The normal stresses will appear as dual variables in the dual problem.
The constrained problem \eqref{eq:primal problem} can be rewritten as an unconstrained minimization of a convex functional.
Let us write as $I$ the indicator function defined as
\[ I(x) = \left\{ \begin{alignedat}{2}
& 0 &\text{ if } x \ge 0, \\
& +\infty &\text{ if } x < 0.
\end{alignedat} \right. \]
We can now define the proper convex function $\mathtt E:V \to \mathbb{R} \cup \{ +\infty \}$ as
\begin{equation}
\label{eq:primal energy convex}
\mathtt E(\bm u) = E(\bm u) + \sum_{e \in \mathcal E^i} I([\bm u]_e \cdot \bm n_e).
\end{equation}
Note that for simplicity, we neglect the equality constraint from \eqref{eq:primal problem} in the exposition of the dual problem.

\begin{prop}
\label{th:existence}
Assume that $\bm g$ is such that $E(\bm u) \to +\infty$ when $\| \bm u\|_V \to +\infty$.
Then, the functional $\mathtt E$ defined in \eqref{eq:primal energy convex} admits a minimizer.
\end{prop}

\begin{remark}
\label{rem:confinement_hyp}
Since the energy \eqref{eq:primal} does not contain any elastic forces, see \eqref{eq:pair energy}, one needs a confinement term $\bm g$ strong enough to avoid cells slipping to infinity.
Section \ref{sec:num results} contains several examples of such confinement terms $\bm g$.
\end{remark}

\begin{proof}[Proof of Proposition \ref{th:existence}]
Let us define $\mathcal V := \{ \bm v \in V ; \ \forall e \in \mathcal E^i, \quad [\bm u]_e \cdot \bm n_e \ge 0 \}$.
$\mathcal V$ is a closed subspace of $V$ and thus is a vector space.
Note that over $\mathcal V$, one has $\mathtt E = E$ and $\mathtt E$ is therefore continuous over $\mathcal V$.
Using the hypothesis and the fact that $\mathcal V$ is finite-dimensional, once can deduce that $\mathtt E$ admits a global minimizer over $\mathcal V$.
\end{proof}

Let $W := \mathbb P^0(\mathcal E^i)$.
We write dual variables as $\bm f \in W^2$.
The dual energy $\mathtt E^*$ is then defined as the convex conjugate of $\mathtt E$ by
\begin{equation}
    \mathtt E^{*}(\bm f) = \sup_{\bm u \in V} \left\{ \langle \bm f, \bm u \rangle - \mathtt E(\bm u) \right\}.
\end{equation}
The dual problem is then defined as 
\begin{equation}
\label{eq:dual problem}
\max_{\bm f \in W^2} \mathtt E^*(\bm f).
\end{equation}
Note that, since $\mathtt E$ has the physical dimension of an energy, $\bm f$ is interpreted as a collection of forces associated to each internal edge, with the corresponding physical dimension.
Let us recall the following classical result.

\begin{prop}[Weak Duality]
\label{th:weak}
If $\bm u \in V$ is feasible for \eqref{eq:primal problem} and $\bm  f$ is feasible for \eqref{eq:dual problem}, then
\[ \mathtt E(\bm u) \ge \mathtt E^*(\bm f). \]
\end{prop}
Note that Proposition \ref{th:weak} shows that, if Proposition \ref{th:existence} applies, then \eqref{eq:dual problem} admits a solution.

\begin{remark}
Since $\mathtt E$ is not strictly convex, one cannot expect to have uniqueness of minimizers.
A similar reasoning applies to $\mathtt E^*$ and \eqref{eq:dual problem}.
\end{remark}

\begin{remark}
\label{rem:confinement_practice}
In practice, if the solver used to compute a minimizer of \eqref{eq:dual problem} does not converge, adding a stronger confining term to $\bm g$ is often enough to obtain a solution of \eqref{eq:dual problem}.
\end{remark}

Let $e \in \mathcal E^i$, we write $\bm f_e = (f^n_e, f^t_e)$.
Assuming \eqref{eq:dual problem} has a feasible solution, the normal stress for a cell $c \in \mathcal M$ and an internal edge $e \in \mathcal E_c \cap \mathcal E^i$ is thus defined as
\begin{equation}
\label{eq:normal stress}
\bm \lambda_{e,c} := |e|^{-1} (\bm n_c \cdot \bm n_e) \left(f_e^n \bm n_e + f_e^t \bm t_e \right). 
\end{equation}

\begin{remark}
The principle of action-reaction is intrinsically taken into account here since for $c,c' \in \mathcal M$ and $e \in \mathcal E_c \cap \mathcal E_{c'} \neq \emptyset$, one has $\bm \lambda_{e,c} = - \bm \lambda_{e,c'}$.
That is due to the fact that one has, for instance, $\bm n_c = \bm n_e$ and $\bm n_{c'} = -\bm n_e$.
\end{remark}

\section{Stress reconstruction}
\label{sec:reconstruction}
Solving \eqref{eq:dual problem} gives us, through \eqref{eq:normal stress}, a normal stress in each internal edge $e \in \mathcal E^i$, which are completed by the normal stresses $\bm g$ imposed as Neumann boundary conditions on $\mathcal E^\partial$.
Let us consider a cell $c \in \mathcal M$ as in Figure \ref{fig:stress reconstruction}.
\begin{figure}[!htp]
\centering
\begin{tikzpicture}
  \coordinate (A) at (0,0);
  \coordinate (B) at (2,0);
  \coordinate (C) at (2.5,2);
  \coordinate (D) at (1.5,3.5);
  \coordinate (E) at (0.5,3);
  \coordinate (F) at (-0.5,1.5);

  \draw[thick] (A) -- (B) -- (C) -- (D) -- (E) -- (F) -- cycle; 
  
  \coordinate (G) at (barycentric cs:A=1/6,B=1/6,C=1/6,D=1/6,E=1/6,F=1/6);
  \draw node at (G) {$\bullet$};
  \draw node[below] at (G) {$\bm x_c$};

  \draw[dashed, thin] (G) -- (A);
  \draw[dashed, thin] (G) -- (B);
  \draw[dashed, thin] (G) -- (C);
  \draw[dashed, thin] (G) -- (D);
  \draw[dashed, thin] (G) -- (E);
  \draw[dashed, thin] (G) -- (F);

  \path (B) -- (A) coordinate[midway] (M);
  \draw[->, thick, red] (M) -- ++(0,-0.5);
  \draw node[red,above] at (M) {$\bm n_c$};
  
  \path (B) -- (C) coordinate[midway] (MBC);
  \node[blue] at (MBC) {$\bullet$};
  \node[blue,right] at (MBC) {$\bm x_e$};
  \node[below] at (2.3, .8){$e$};
  
  \draw[->, thick, red] (MBC) -- ++(-2/4.5,1/2/4.5);
  \node[above,red] at (MBC) {$\bm n_e$};
\end{tikzpicture}
\caption{Stress reconstruction in an internal cell.}
\label{fig:stress reconstruction}
\end{figure}
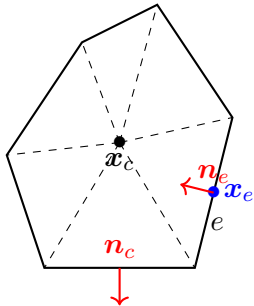
We define as $H(\div;c)^2 := \{ \sigma \in L^2(c)^{2 \times 2}; \div(\sigma) \in L^2(c)^2 \}$.
The associated norm is defined for $\sigma \in H(\div;c)$ by $\| \sigma \|_{H(\div;c)}^2 := \| \sigma \|_{L^2(c)}^2 + \| \div(\sigma) \|_{L^2(c)}^2$.
The idea of the reconstruction is to compute in each cell $c \in \mathcal M$, a field $\sigma_c \in H(\div;c)^2$, so that $\div(\sigma_c) = \bm 0$ in $c$, since no volume load is applied to $\Omega$.
This is performed by solving local problems with finite elements, as described in the following.
Note that, if volumes loads were applied to $\Omega$, one could modify the continuity equation in order to take the volumes loads into account.

\subsection{Geometric setup}
Let $c \in \mathcal M$.
We write as $\bm n_c$ the outward unit normal to $\partial c$ and the barycentre of the internal edge $e \in \mathcal E^i$ as $\bm x_e$.
Since all cells are convex, the boundary $\partial c$ of $c \in \mathcal M$ is Lipschitz \cite{grisvard2011elliptic}.
Also, since any cell $c \in \mathcal M$ is convex, it can be triangulated by choosing an edge $e \in \mathcal E_c$ and the center $\bm x_c$, as sketched in Figure \ref{fig:stress reconstruction}.
Let us write as $\mathcal M_c$ this triangulation of the cell $c$.

Let $\mathfrak V := \mathbb {RT}_0(\mathcal M_c)^2$, where $\mathbb {RT}_0(\mathcal M_c)$ is the space of lowest-order Raviart--thomas elements, see \cite{MR4242224}.
We will use the Neumann boundary condition $\bm g$ and the solutions of \eqref{eq:dual problem} $\bm \lambda$ as Dirichlet boundary conditions.
We define the solutions set as
\[ \mathfrak V_D := \{ \sigma \in \mathfrak V; \, \forall e \in \mathcal E_c \cap \mathcal E^\partial, \, \sigma(\bm x_e) \bm n_c = \bm g_e \text{ and } \forall e \in \mathcal E_c \cap \mathcal E^i, \, \sigma(\bm x_e) \bm n_c = \bm \lambda_{e,c} \}.  \]
The associated homogeneous space is defined as
\[ \mathfrak V_0 := \{ \sigma \in \mathfrak V; \, \forall e \in \mathcal E_c, \, \sigma(\bm x_e) \bm n_c = \bm 0 \}. \]
Note that, for $\sigma \in \mathfrak V$, one does not have $\sigma^\mathsf{T} = \sigma$.
This condition will be imposed weakly in the following through a Lagrange multiplier.
Let us define the space of Lagrange multipliers $Q := \mathbb P^0(\mathcal M_c)$, which is the space of piecewise constant functions.

\subsection{Discrete problem}
The formulation of this problem has been inspired by \cite{arnold1984peers,arnold2007mixed}.
For $\sigma, \tau \in H(\div;c)^2$, we define the bilinear form
\[ a(\sigma, \tau) := \int_c \mathrm{div}(\sigma) \cdot \mathrm{div}(\tau). \]
This bilinear form consists in a least-square formulation of the equation $\mathrm{div}(\sigma) = \bm 0$.
For $\sigma \in H(\div;c)^2$ and $q \in L^2(c)$, we define the bilinear form
\[ b(\sigma, q) := \int_c \sigma : j(q), \]
where $j(q) := \begin{pmatrix}
0 & q \\
-q & 0
\end{pmatrix}$ is a generic skew-symmetric matrix.
Note that one has $|j(q)| = \sqrt{2} |q|$ when considering the Frobenius norm on matrices.
This bilinear form is used to weakly impose $\sigma^\mathsf{T} = \sigma$ in $c$.
The reconstruction consists in searching for $(\sigma, p) \in \mathfrak V_D \times Q$ such that
\begin{equation}
\label{eq:reconstruction}
\begin{aligned}
a(\sigma, \tau) + b(\tau, p) = 0, & \quad \forall \tau \in \mathfrak V_0, \\
b(\sigma, q) = 0, & \quad \forall q \in Q.
\end{aligned}
\end{equation}

\begin{prop}
\label{th:recon}
There exists a unique solution of \eqref{eq:reconstruction}.
\end{prop}

\begin{remark}
Note that Proposition \ref{th:recon} shows that from a given $\bm \lambda \in W^2$, solution of \eqref{eq:dual problem}, one can compute a unique $\sigma \in \mathfrak V_D$.
But since $\bm \lambda \in W^2$, solution of \eqref{eq:dual problem}, is not unique, the reconstructed stress $\sigma$ is not unique either.
\end{remark}

\begin{remark}
\label{rem:torque}
The use of $\mathbb{RT}_0$ elements implies a subtle consistency requirement regarding rotations.
In standard elasticity, the symmetry of the stress tensor $\sigma = \sigma^T$ (which we impose weakly via \eqref{eq:reconstruction}) is equivalent to the local balance of angular momentum (torque balance).
However, since our discrete model (Section \ref{sec:model}) neglects rotational equilibrium equations, the discrete forces/fluxes $\bm \lambda$ might not exactly sum to zero torque on each cell.
\end{remark}

\begin{proof}[Proof of Proposition \ref{th:recon}]
In order to prove that \eqref{eq:reconstruction} admits a unique solution, we want to use the BNB theorem for mixed systems, see \cite[Theorem 2.34]{ern_guermond}.
Let us first lift the Dirichlet boundary conditions.
Using the surjectivity of the trace operator, see \cite[Theorem~4.15]{MR4242224}, there exists $\mu \in H(\div;c)^2$, such that for all $e \in \mathcal E_c \cap \mathcal E^\partial$, $\int_e \mu \bm n = |e| \bm g_e$  and for all $e \in \mathcal E_c \cap \mathcal E^i$, $\int_e \mu \bm n = |e| \bm \lambda_{e,c}$.
For $\tau \in \mathfrak V_0$, let us define the linear form 
\[ F(\tau) := -\int_c \div(\mu) \cdot \div(\tau), \]
which is continuous over $\mathfrak V_0$ and for $q \in Q$,
\[ G(q) := - \int_c \mu : j(q), \]
which is linear continuous over $Q$.
We now search for $(\sigma_0, p) \in \mathfrak V_0 \times Q$,
\[ \begin{aligned}
a(\sigma_0, \tau) + b(\tau, p) = L(\tau), & \quad \forall \tau \in \mathfrak V_0, \\
b(\sigma_0, q) = G(q), & \quad \forall q \in Q.
\end{aligned} \]

Let us define $H_0(\div;c)^2 := \{ \sigma \in H(\div;c)^2; \sigma \bm n = \bm 0 \text{ on } \partial \Omega \}$.
We propose to show first that $a$ is coervice over $Z := H_0(\div;c)^2 \cap \{\sigma \in L^2(\Omega)^{2 \times 2} ; \, \mathrm{curl}(\sigma) = \bm 0 \}$, where the 2d $\mathrm{curl}$ operator is applied row-wise.
To prove that coercivity, we need to show that there exists $C > 0$, such that for any $\sigma \in Z$, one has
\[ \| \sigma \|_{L^2(c)} \le C \| \div(\sigma) \|_{L^2(c)}. \]
Let us reason by density.
Let $\varphi \in C^\infty_c(c)^{2 \times 2}$, such that $\mathrm{curl}(\varphi) = \bm 0$.
Since $\varphi \in H^1_0(c)^{2 \times 2}$, we can apply the Poincar\'e inequality and there exits $C > 0$,
\[ \|\varphi \|_{L^2(c)} \le C \|\nabla \varphi \|_{L^2(c)}. \]
Using \cite[Theorem~2.3]{costabel1999maxwell} and the fact that $\mathrm{curl}(\varphi) = \bm 0$, one has
\[ \| \varphi \|_{L^2(c)} \le C \| \div(\varphi) \|_{L^2(c)}. \]
We now consider a sequence $(\varphi_n)_n \in C^\infty_c(c)^{2 \times 2}$, such that $\mathrm{curl}(\varphi_n) = \bm 0$ and $\sigma \in Z$ such that $\| \varphi_n - \sigma \|_{H(\div;c)} \to 0$, as $n \to +\infty$.
Thus, as $n \to +\infty$, $\| \varphi_n - \sigma \|_{L^2(c)} \to 0$ and $\| \div(\varphi_n - \sigma) \|_{L^2(c)} \to 0$.
One has
\[ \begin{aligned}
\| \div(\sigma) \|_{L^2(c)} & = \| \div(\varphi_n) |_{L^2(c)} + \| \div(\varphi_n - \sigma) |_{L^2(c)}, \\
& \ge \frac1C \| \varphi_n \|_{L^2(c)} + \| \div(\varphi_n - \sigma) |_{L^2(c)} \mathop{\longrightarrow}_{n \to +\infty} \frac1C \| \sigma \|_{L^2(c)}. 
\end{aligned} \]
Therefore,
\[ \| \sigma \|_{H(\div;c)}^2 = \| \div(\sigma) \|^2_{L^2(c)} + \| \sigma \|^2_{L^2(c)} \le \left(1 + C^2 \right) \| \div(\sigma) \|^2_{L^2(c)}. \]
Thus, one has
\[ a(\sigma, \sigma) = \| \div(\sigma) \|^2_{L^2(c)} \ge  \frac1{1 + C^2} \| \sigma \|_{H(\div;c)}^2, \]
which shows the coercivity over $Z$.
To finish, we note that $\mathfrak V_0$ is conforming in $Z$ in the sense that $\mathfrak V_0 \subset Z$.

In order to apply the BNB theorem for mixed formulation \cite[Theorem~2.34]{ern_guermond}, it remains to show the second condition of the theorem.
Let $q \in Q$ and $\tau \in \mathfrak V_0$.
One has
\[ \begin{aligned}
b(\tau, q) = \sum_{T \in \mathcal M_c} \int_T \tau : j(q) & = \sum_{T \in \mathcal M_c} |T| \skew(\left<\tau\right>_T) : j(q), \\
& \le \sum_{T \in \mathcal M_c} |T| \left|\skew(\left<\tau\right>_T)\right| \sqrt{2} |q|, \\
& \le \sqrt{2} \left( \sum_{T \in \mathcal M_c} |T| \left|\skew(\left<\tau\right>_T)\right|^2 \right)^{\frac12} \left( \sum_{T \in \mathcal M_c} |T| |q|^2 \right)^{\frac12}, \\
& = \sqrt{2} \left( \sum_{T \in \mathcal M_c} \| \skew(\left< \tau \right>_T) \|^2_{L^2(T)} \right)^{\frac12} \|q \|_{L^2(c)}
\end{aligned} \]
where $\left<\tau \right>_T$ is the average over $T$,
$\skew$ is the skew-symmetric part of a matrix, and we used the Cauchy--Schwarz inequality twice.
Note that since $\tau$ is affine in each $T \in \mathcal M_c$, one has $\left<\tau \right>_T = \tau(\bm x_T)$, for all $T \in \mathcal M_c$.

Let us now consider $\tau \in \mathfrak V_0$ such that for all $T \in \mathcal M_c$, $\skew(\tau(\bm x_T)) = j(q_{|T})$.
Let us explain why this is always possible when considering the kind of meshes $\mathcal M_c$ as in Figure \ref{fig:stress reconstruction}.
Let $n \in \mathbb N$ be the number of cells in $\mathcal M_c$.
We thus have $6n$ dofs for $\tau \in \mathfrak V$.
For $\tau \in \mathfrak V_0$, the boundary conditions imposed on $\partial c$ consist in $2n$ equations.
There are as many cells as internal edges, see Figure \ref{fig:stress reconstruction}.
The normal continuity conditions at internal edges consist in $2n$ equations.
Finally imposing $\skew(\tau(\bm x_T)) = j(q_{|T})$, for all $T \in \mathcal M_c$, consists in $n$ equations.
To conclude, we have $5n$ equations for $6n$ dofs.

With such a $\tau \in \mathfrak V_0$, the two Cauchy--Schwarz inequalities above are actually equalities.
One thus has
\[ b(\tau, q) = \sqrt{2} \left( \sum_{T \in \mathcal M_c} \| \skew(\left< \tau \right>_T) \|^2_{L^2(T)} \right)^{\frac12} \|q \|_{L^2(c)}. \]
Also, one has
\[ \sum_{T \in \mathcal M_c} \| \skew(\left< \tau \right>_T) \|^2_{L^2(T)} \le \sum_{T \in \mathcal M_c} \| \left< \tau \right>_T \|^2_{L^2(T)} \le \sum_{T \in \mathcal M_c} \| \tau \|^2_{L^2(T)} = \| \tau \|^2_{L^2(c)} \le \| \tau \|^2_{H(\div;c)}. \]
Therefore,
\[ \frac{b(\tau,q)}{\| \tau \|_{H(\div;c)}} \le \frac{b(\tau,q)}{\left(\sum_{T \in \mathcal M_c} \| \skew(\left< \tau \right>_T) \|^2_{L^2(T)} \right)^{\frac12}} = \sqrt{2} \| q \|_{L^2(c)}. \]
The supremum over $\varphi \in \mathfrak V_0 $ is thus achieved and one has
\[ \inf_{q \in Q} \sup_{\varphi \in \mathfrak V_0} \frac{b(\tau,q)}{\| \varphi \|_{H(\div;c)} \| q \|_{L^2(c)}} \ge \sqrt{2}, \]
which finishes the proof.
\end{proof}

\section{Numerical results}
\label{sec:num results}
We first describe how the normal stresses are computed in practice.
Then, we present a few tests to ensure the consistency of the proposed method.
Finally, we show the robustness of the method on unstructured meshes.
Note that, even though our methodology only requires convex polygonal cells, the numerical examples below use Voronoi meshes as this is a convenient way to generate structured/unstructured polygonal meshes with convex cells.

\subsection{Linear program}
Let us rewrite \eqref{eq:primal problem} as a linear program.
This allows us to use very efficient linear solvers.
For each internal edge $e \in \mathcal E^i$, we add a new variable $v_e$ such that $v_e \ge |[\bm u]_e \cdot \bm t_e|$.
Note that this inequality constraint can be written with two linear constraints.
Let us write $v := (v_e)_e \in W$.
Therefore, the linear program is now
\begin{equation}
\label{eq:primal linear program}
\begin{aligned}
\min_{(\bm u, v) \in V \times W} \quad & \sum_{e \in \mathcal E^i} s_T |e| v_e -  E_\text{ext}(\bm u)\\
\text{subject to} \quad & \forall e \in \mathcal E^i, \quad [\bm u]_e \cdot \bm n_e \ge 0, \\
& \forall e \in \mathcal E^i, \quad v_e \ge [\bm u]_e \cdot \bm t_e, \\
& \forall e \in \mathcal E^i, \quad v_e \ge -[\bm u]_e \cdot \bm t_e, \\
 & \sum_{i=1}^N \bm u_i = 0.
\end{aligned}
\end{equation}

\subsection{Implementation}
The implementation is performed in {\em Python}.
The package {\em pyvoro}, which uses the library {\em Voro++} \cite{rycroft2009voro++} is used to generate the Voronoi mesh.
The package {\em cvxopt} \cite{vandenberghe2010cvxopt} is used to solve \eqref{eq:primal linear program} and provide the normal stresses by computing the dual linear program.
Finally, {\em Firedrake} \cite{FiredrakeUserManual} is used to solve \eqref{eq:reconstruction}.

\subsection{Verification tests}
For the first test case, we consider $\Omega = (0,1)^2$ and the mesh sketched in Figure \ref{fig:domain test}.
\begin{figure}[!htp]
\centering
\begin{subfigure}[b]{0.45\textwidth}
        \centering
        \includegraphics[scale=.3]{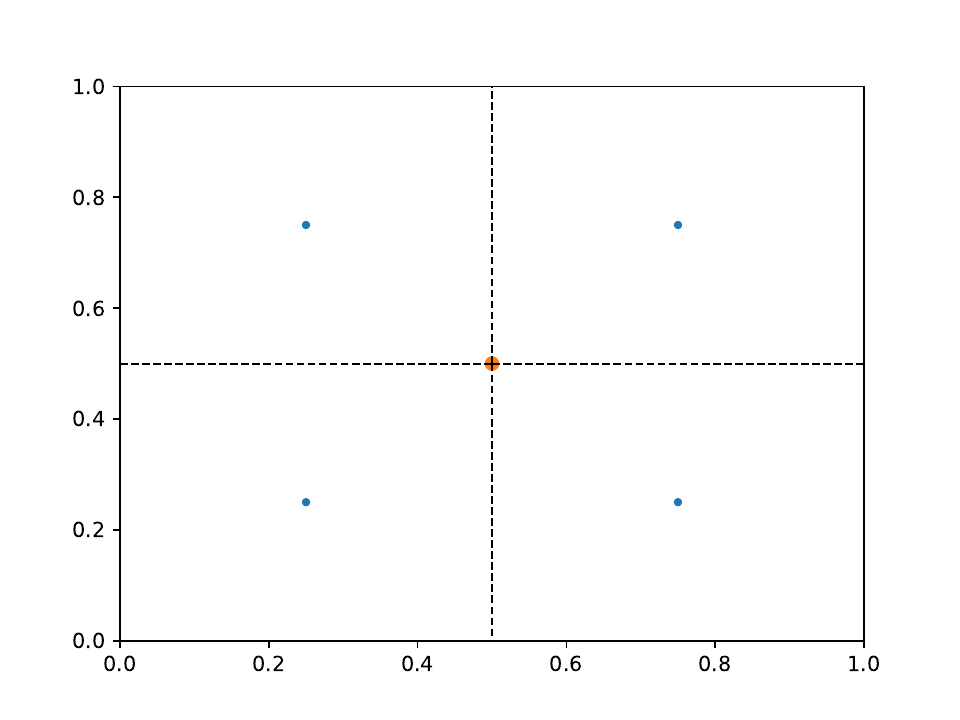}
        \
\caption{Mesh $\mathcal M$ of $\Omega$}
        \label{fig:domain test}
    \end{subfigure}
    \hfill
    \begin{subfigure}[b]{0.5\textwidth}
        \centering
        \includegraphics[scale=.25]{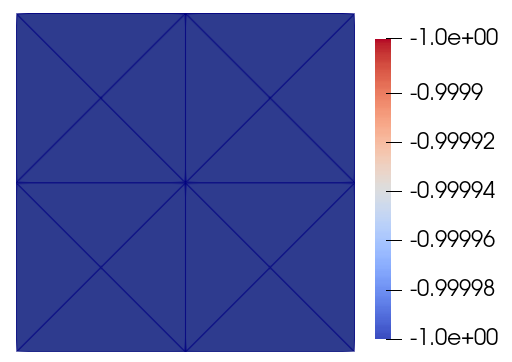}
        \caption{Reconstructed stresses $\sigma_{ij}$.}
        \label{fig:stress test}
    \end{subfigure}
    \caption{Verification test 1}
\end{figure}
The Tresca friction coefficient is taken as $s_T := 10$.
Let $S = - \begin{pmatrix}
1 & 1 \\ 1 & 1 \\
\end{pmatrix}$.
The boundary condition applied to the domain is $\bm g := S\bm n$.
Figure \ref{fig:stress test} shows the resconstructed stress $\sigma$.
Note that $\sigma_{11}=\sigma_{22} = \sigma_{12}=\sigma_{21}$ in Figure \ref{fig:stress test} and therefore only one picture is shown.
The result we obtain is consistent with the imposed boundary condition since we expect $\sigma = S$ for this test case.

For the second test case, we consider the same domain, material parameters and boundary conditions.
However, we consider the mesh sketched in Figure \ref{fig:domain test 2}.
\begin{figure}[!htp]
\centering
\begin{subfigure}[b]{0.45\textwidth}
        \centering
        \includegraphics[scale=.2]{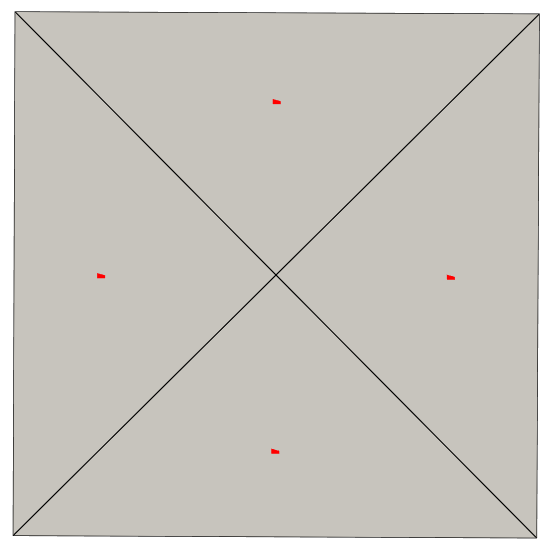}
        \
\caption{Mesh $\mathcal M$ of $\Omega$}
        \label{fig:domain test 2}
    \end{subfigure}
    \hfill
    \begin{subfigure}[b]{0.45\textwidth}
        \centering
        \includegraphics[scale=.4]{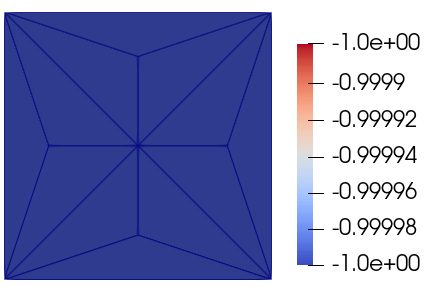}
        \caption{Reconstructed stress $\sigma_{ij}$.}
        \label{fig:stress test 2}
    \end{subfigure}

    \caption{Verification test 2}
\end{figure}

Figure \ref{fig:stress test 2} shows the reconstructed stress $\sigma$.
The result we obtain is consistent with the imposed boundary condition since we expect $\sigma = S$ for this test case too.

\subsection{Shear test}
We consider $\Omega = (0,1)^2$, the mesh shown in Figure \ref{fig:brick wall mesh} and the boundary conditions sketched in Figure \ref{fig:brick wall BC}.
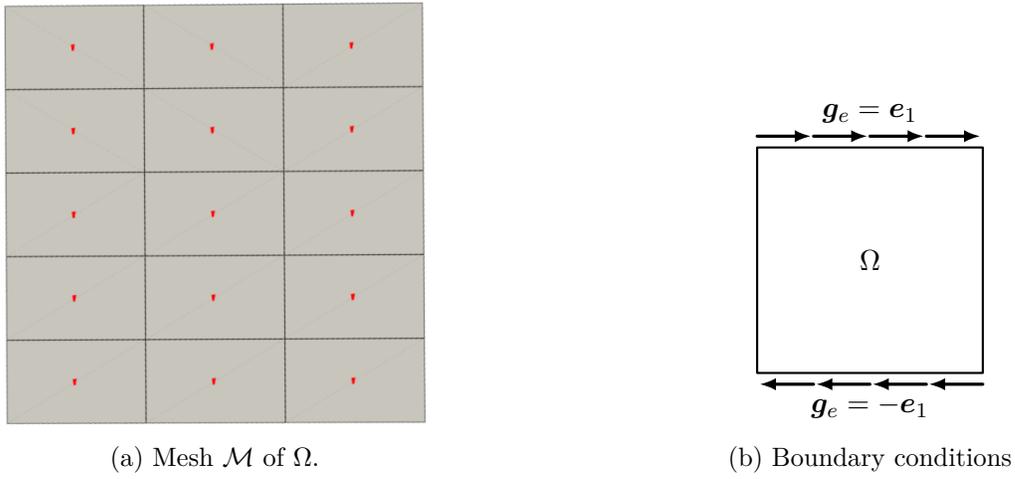
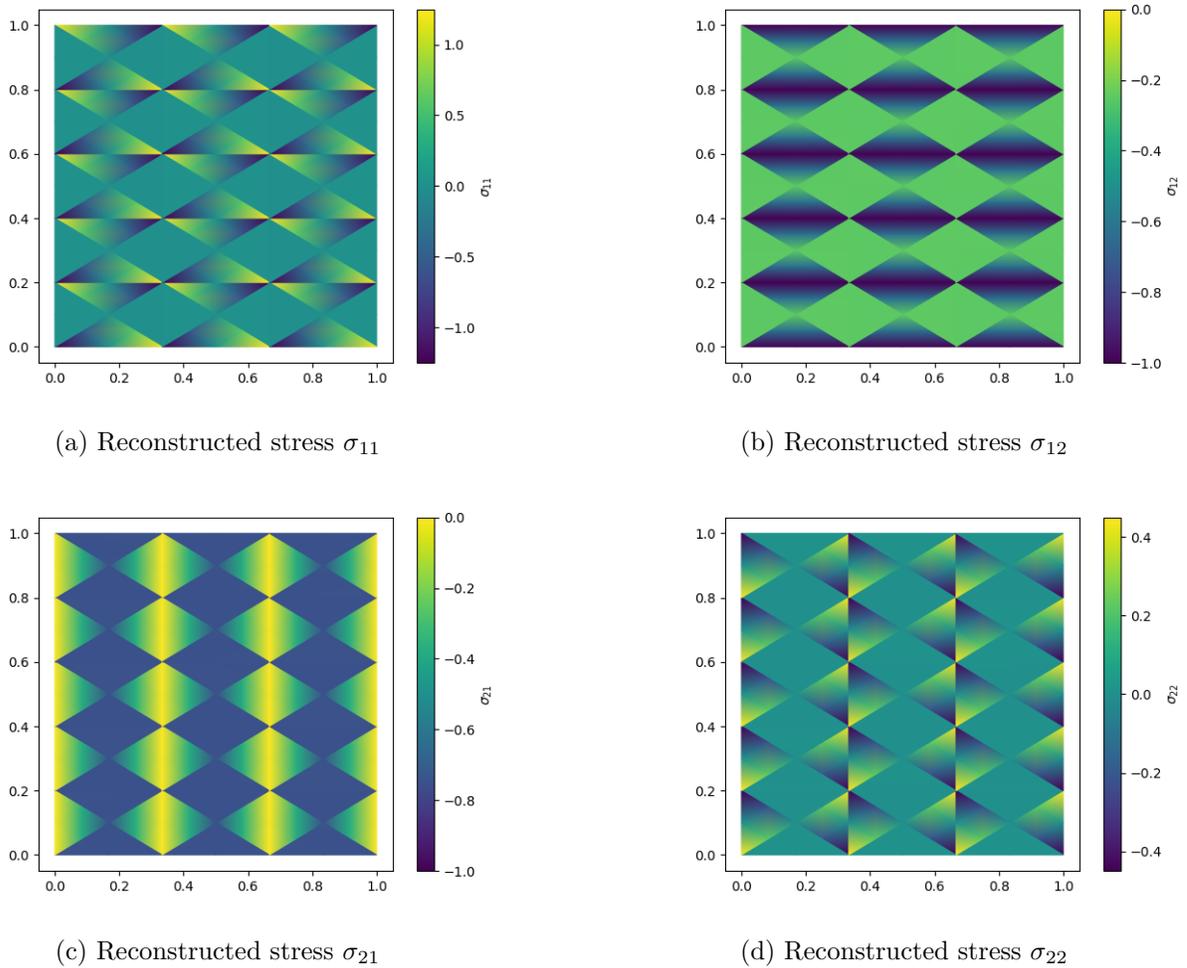
\begin{figure}[!htp]
\centering
\begin{subfigure}[b]{0.45\textwidth}
        \centering
        \includegraphics[scale=.3]{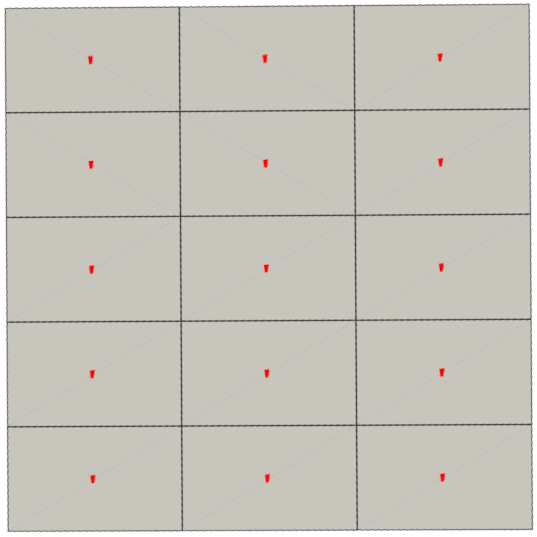}
	\caption{Mesh $\mathcal M$ of $\Omega$.}
     \label{fig:brick wall mesh}
    \end{subfigure}
    \hfill
    \begin{subfigure}[b]{0.5\textwidth}
        \centering
        
\begin{tikzpicture}[scale=3, line cap=round, line join=round, >=Latex]
\draw[thick] (0,0) rectangle (1,1);
\node at (.5,.5) {$\Omega$};

\foreach \x in {0,0.25,0.5,0.75}{
  \draw[-{Latex[length=2.2mm]},very thick] (\x,1.05) -- ++(0.24,0);
}
\node at (0.5,1.15) {$\bm g_e = \bm e_1$};

\foreach \x in {0.25,0.5,0.75,1}{
  \draw[-{Latex[length=2.2mm]},very thick] (\x,-0.05) -- ++(-0.24,0);
}
\node at (0.5,-0.15) {$\bm g_e = -\bm e_1$};
\end{tikzpicture}            
        
        \caption{Boundary conditions}
        \label{fig:brick wall BC}
    \end{subfigure}
    \caption{Shear test}
\end{figure}
We write the usual direct orthonormal basis of $\mathbb R^2$ as $(\bm e_1, \bm e_2)$.
This test consists in shear loading a brick wall.
Note that the weight of the bricks is ignored here.
The computed stresses are presented in Figure \ref{fig:brick wall stress}.
\begin{figure}[!htp]
\centering
    \begin{subfigure}[b]{0.45\textwidth}
        \centering
        \includegraphics[scale=.5]{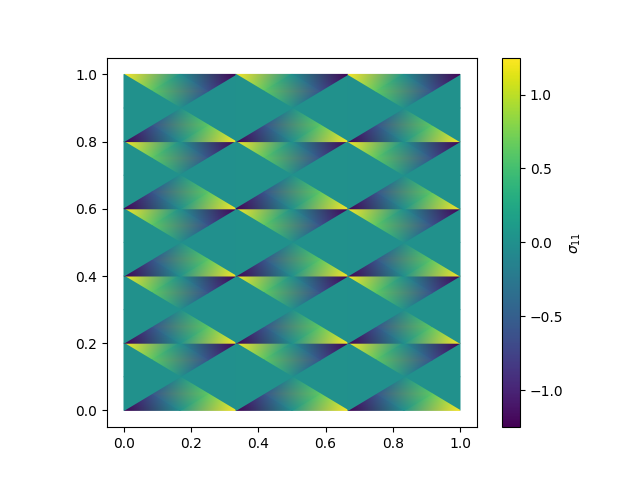}
        \caption{Reconstructed stress $\sigma_{11}$}
    \end{subfigure}
    \hfill
    \begin{subfigure}[b]{0.45\textwidth}
        \centering
        \includegraphics[scale=.5]{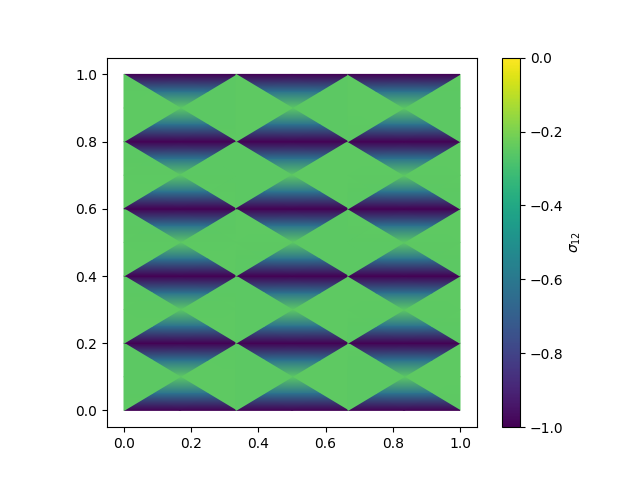}
        \caption{Reconstructed stress $\sigma_{12}$}
    \end{subfigure} \\
    \begin{subfigure}[b]{0.45\textwidth}
        \centering
        \includegraphics[scale=.5]{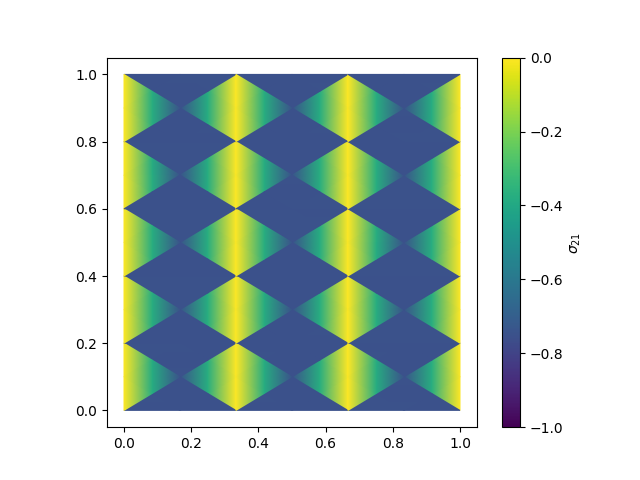}
        \caption{Reconstructed stress $\sigma_{21}$}
    \end{subfigure}
    \hfill
    \begin{subfigure}[b]{0.45\textwidth}
        \centering
        \includegraphics[scale=.5]{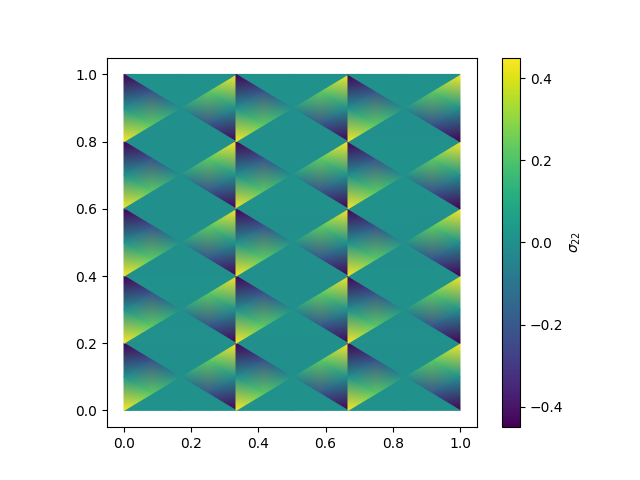}
        \caption{Reconstructed stress $\sigma_{22}$}
    \end{subfigure}
    \caption{Shear test}
    \label{fig:brick wall stress}	
\end{figure}
Note that the stresses are discontinuous at the interfaces between bricks since one only has $\sigma \in H(\div;c)$, for all $c \in \mathcal M$.

\subsection{Homogeneous compression test}
We consider $\Omega = (0,1)^2$ and the mesh sketched in Figure \ref{fig:domain voro}.
\begin{figure}[!htp]
	\centering
    \includegraphics[scale=.3]{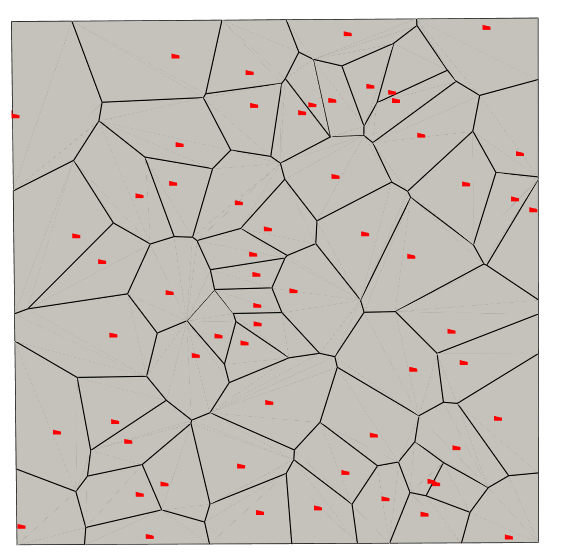}
	\caption{Homogeneous compression test: mesh $\mathcal M$ of $\Omega$.}
     \label{fig:domain voro}
\end{figure}
It contains 60 Voronoi cells.
The boundary conditions consists in imposing $\bm g := -\bm n$ on $\partial \Omega$.
The results are shown in Figure \ref{fig:stress hom compression}.
\begin{figure}[!htp]
\centering
    \begin{subfigure}[b]{0.45\textwidth}
        \centering
        \includegraphics[scale=.5]{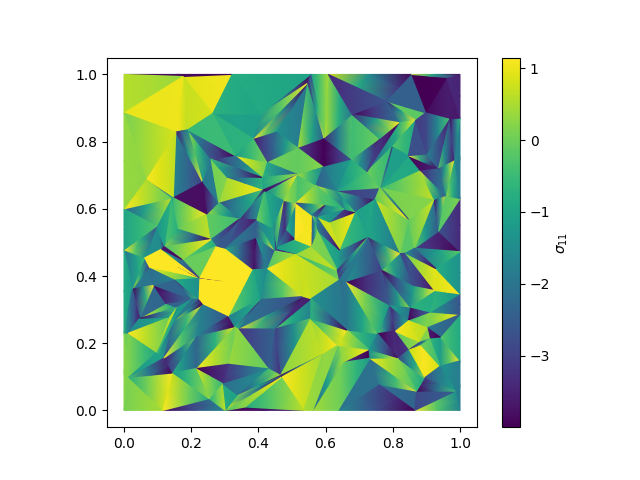}
        \caption{Reconstructed stress $\sigma_{11}$}
    \end{subfigure}
    \hfill
    \begin{subfigure}[b]{0.45\textwidth}
        \centering
        \includegraphics[scale=.5]{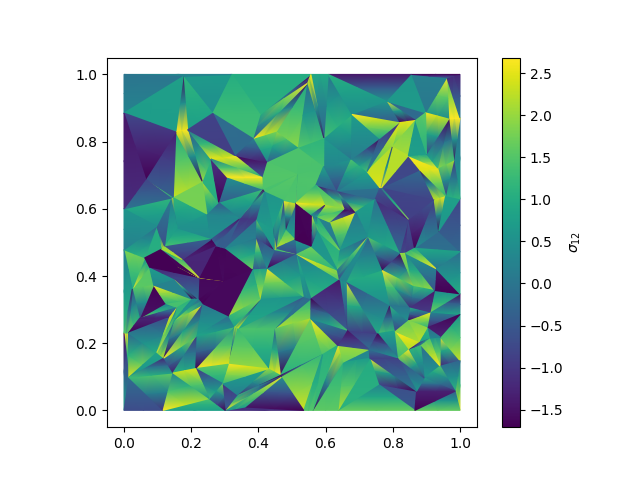}
        \caption{Reconstructed stress $\sigma_{12}$}
    \end{subfigure} \\
    \begin{subfigure}[b]{0.45\textwidth}
        \centering
        \includegraphics[scale=.5]{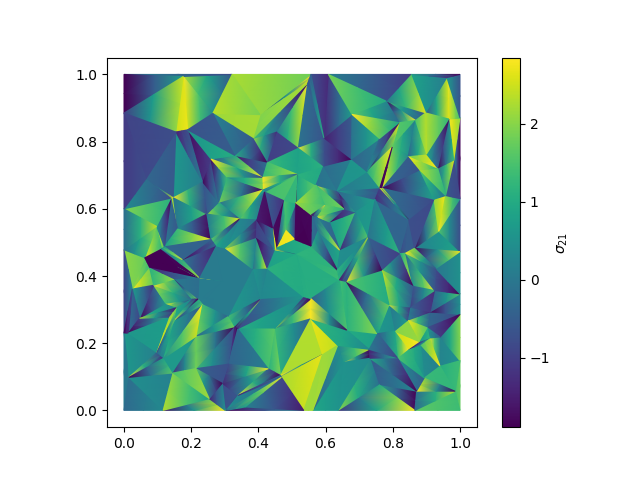}
        \caption{Reconstructed stress $\sigma_{21}$}
    \end{subfigure}
    \hfill
    \begin{subfigure}[b]{0.45\textwidth}
        \centering
        \includegraphics[scale=.5]{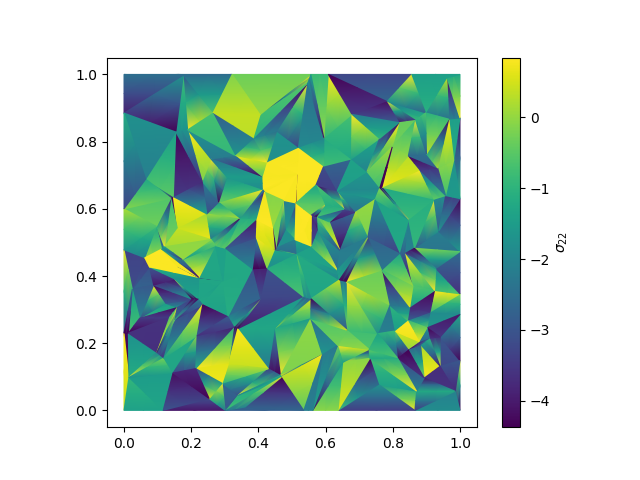}
        \caption{Reconstructed stress $\sigma_{22}$}
    \end{subfigure}
    \caption{Homogeneous compression test}
    \label{fig:stress hom compression}	
\end{figure}
We notice that the recomputed stress $\sigma$ in Figure \ref{fig:stress hom compression} is not homogeneous.

\section{Discussion and Future Work}
\label{sec:discussion}
The numerical results presented in Section \ref{sec:num results} highlight several key features of the proposed model which warrant further discussion in light of the physical theory of granular packings.

In the homogeneous compression test (Figure \ref{fig:stress hom compression}), we observed that the solver identifies a non-homogeneous stress state, despite the boundary conditions and domain being consistent with a uniform stress state $\sigma = -I_2$.
As discussed in Section \ref{sec:static_indet}, this is a direct consequence of the static indeterminacy (or hyperstaticity) of jammed packings.
The uniform solution $\sigma = -I_2$ is indeed a valid equilibrium solution (as verified by re-computing forces), but it is just one point in a larger null space of valid states.
Our optimization approach selects a specific solution from this manifold, namely a global minimizer for the frictional energy \eqref{eq:primal}.
This selection principle mimics nature's tendency to settle into specific energy-minimized configurations, even if there are other static equilibria.
Furthermore, it is well known in optimization theory and statistics that $L^1$ minimization promotes sparsity \cite{Tibshirani1996, Donoho2006}.
In the context of granular materials, this mathematical property has a direct physical interpretation: the formation of {\em force chains} \cite{Majmudar2005}.
By minimizing the sum of frictional work terms (an $L^1$-type objective), the system naturally selects a sparse network of load-bearing paths, leaving the majority of the domain in a state of low or zero stress.
This aligns with experimental observations of heterogeneous force networks in jammed packings \cite{Majmudar2005}.

Regarding torque balance, in the shear test (Figure \ref{fig:brick wall stress}), we noted that $\sigma_{12} \neq \sigma_{21}$.
This asymmetry is expected and physically consistent with our modeling choices.
In standard limit analysis, one enforces both force and torque balance on rigid blocks.
However, in our simplification, we explicitly exclude rotational degrees of freedom. 
Consequently, our ``equilibrium" satisfies weak force balance but not local moment balance.
The skew-symmetric component of the reconstructed stress tensor effectively accounts for the residual torques in the system.

Another counter-intuitive result from the compression test is the presence of cells in tension (positive principal stresses), despite the global compressive loading.
This is not a numerical artifact but a physical phenomenon observed in frictional granular assemblies.
Friction allows for ``interlocking" or ``bridging" effects where local force chains can shield certain regions or induce local tension perpendicular to the primary load path.
This local tension is a critical precursor to fracture initiation in geomaterials \cite{horii1985compression}.

Finally, a natural extension of this work is to restore full mechanical consistency by including rotational degrees of freedom $\theta_c$.
This would require:
\begin{enumerate}
    \item \textbf{Updated Constraints:} With rotations, the normal displacement varies linearly along an edge. To ensure non-interpenetration along the entire edge, it is necessary and sufficient to enforce the constraint at both endpoints of the edge.
This effectively doubles the number of inequality constraints.
    \item \textbf{Complex Objective Function:} Similarly, the tangential relative displacement varies linearly along the edge. The frictional work is the integral of the absolute value of this linear function. While more complex than the constant case, this convex functional can still be linearized (e.g., by splitting the integral at the zero-crossing of the displacement), allowing us to recast as a finite dimensional convex optimization problem.
\end{enumerate}
Implementing these features to recover a fully moment-balanced (symmetric stress) limit analysis remains a challenging avenue for future research.
Beyond moment balance, adopting a more realistic friction model such as Coulomb friction is also desirable. \label{disc:coulomb} While Tresca friction allows for a global energy minimization formulation (since the threshold $s_T$ is fixed), Coulomb friction introduces a coupling between the tangential threshold and the normal stress ($|\sigma_t| \le \mu |\sigma_n|$) resulting in a \textit{Quasi-Variational Inequality} \cite{baiocchi1984variational}. This dependency typically breaks the convexity of the associated energy functional, potentially requiring iterative fixed-point algorithms or incremental evolution schemes rather than a single static optimization step.

\section*{Acknowledgements}
The authors would like to thank the anonymous referee for their useful comments which helped improve this manuscript. SV was supported in part by NSF grants DMS-2108124 and DMS-2511503.

\section*{Data Availability Statement}
The code is available at \url{https://github.com/marazzaf/masonry_code.git}.

\bibliographystyle{plain}
\bibliography{references}

\end{document}